\documentclass[11pt]{amsart}
\usepackage{epsfig}
\usepackage{graphicx}
\usepackage{amscd}
\usepackage{amsmath}
\usepackage{amsxtra}
\usepackage{amsfonts}
\usepackage{amssymb}

\newtheorem{theorem}{Theorem}[section]
\newtheorem{corollary}[theorem]{Corollary}
\newtheorem{lemma}[theorem]{Lemma}
\newtheorem{proposition}[theorem]{Proposition}

\theoremstyle{definition}
\newtheorem{definition}[theorem]{Definition}
\newtheorem{remark}[theorem]{Remark}

\newtheorem*{question}{Question}
\newtheorem{example}[theorem]{Example}
\theoremstyle{remark}

\renewcommand{\theclaim}{\textup{\theclaim}}

\newtheorem*{acknowledgements}{Acknowledgements}

\numberwithin{equation}{section}

\def\openone

{\mathchoice

{\hbox{\upshape \small1\kern-3.3pt\normalsize1}}

{\hbox{\upshape \small1\kern-3.3pt\normalsize1}}

{\hbox{\upshape \tiny1\kern-2.3pt\SMALL1}}

{\hbox{\upshape \Tiny1\kern-2pt\tiny1}}}

\makeatletter

\newbox\ipbox

\newcommand{\diracb}[1]{\left\langle #1\mathrel{\mathchoice

{\setbox\ipbox=\hbox{$\displaystyle \left\langle\mathstrut
#1\right.$}

\vrule height\ht\ipbox width0.25pt depth\dp\ipbox}

{\setbox\ipbox=\hbox{$\textstyle \left\langle\mathstrut
#1\right.$}

\vrule height\ht\ipbox width0.25pt depth\dp\ipbox}

{\setbox\ipbox=\hbox{$\scriptstyle \left\langle\mathstrut
#1\right.$}

\vrule height\ht\ipbox width0.25pt depth\dp\ipbox}

{\setbox\ipbox=\hbox{$\scriptscriptstyle \left\langle\mathstrut
#1\right.$}

\vrule height\ht\ipbox width0.25pt depth\dp\ipbox}

}\right. }

\newcommand{\dirack}[1]{\left. \mathrel{\mathchoice

{\setbox\ipbox=\hbox{$\displaystyle \left.\mathstrut
#1\right\rangle$}

\vrule height\ht\ipbox width0.25pt depth\dp\ipbox}

{\setbox\ipbox=\hbox{$\textstyle \left.\mathstrut
#1\right\rangle$}

\vrule height\ht\ipbox width0.25pt depth\dp\ipbox}

{\setbox\ipbox=\hbox{$\scriptstyle \left.\mathstrut
#1\right\rangle$}

\vrule height\ht\ipbox width0.25pt depth\dp\ipbox}

{\setbox\ipbox=\hbox{$\scriptscriptstyle \left.\mathstrut
#1\right\rangle$}

\vrule height\ht\ipbox width0.25pt depth\dp\ipbox}

} #1\right\rangle}

\newcommand{\cj}[1]{\overline{#1}}

\newcommand{\bz}{\mathbb{Z}}

\newcommand{\br}{\mathbb{R}}
\newcommand{\bc}{\mathbb{C}}

\newcommand{\bn}{\mathbb{N}}

\def\blfootnote{\xdef\@thefnmark{}\@footnotetext}

\renewcommand{\mod}{\operatorname{mod}}

\input xy
\xyoption{all}
\usepackage{amssymb}

\def\H{\mathcal{H}}

\def\-{^{-1}}

\def\D{\mathcal{D}}

\def\ty{\emptyset}

\begin{document}

\title[Divergence of mock and scrambled Fourier series]{Divergence of mock and scrambled Fourier series on fractal measures}
\author{Dorin Ervin Dutkay}
\blfootnote{}
\address{[Dorin Ervin Dutkay] University of Central Florida\\
	Department of Mathematics\\
	4000 Central Florida Blvd.\\
	P.O. Box 161364\\
	Orlando, FL 32816-1364\\
U.S.A.\\} \email{Dorin.Dutkay@ucf.edu}

\author{Deguang Han}
\address{[Deguang Han] University of Central Florida\\
    Department of Mathematics\\
    4000 Central Florida Blvd.\\
    P.O. Box 161364\\
    Orlando, FL 32816-1364\\
U.S.A.\\} \email{Deguang.Han@ucf.edu}

\author{Qiyu Sun}
\address{[Qiyu Sun] University of Central Florida\\
    Department of Mathematics\\
    4000 Central Florida Blvd.\\
    P.O. Box 161364\\
    Orlando, FL 32816-1364\\
U.S.A.\\} \email{qiyu.sun@ucf.edu}
\thanks{}

\begin{abstract}
  We study divergence properties of Fourier series on Cantor-type fractal measures, also called mock Fourier series. We show that in some cases the $L^1$-norm of the corresponding Dirichlet kernel grows exponentially fast, and therefore the Fourier series are not even pointwise convergent. We apply these results to the Lebesgue measure to show that a certain rearrangement of the exponential functions, with affine structure, which we call scrambled Fourier series, have a corresponding Dirichlet kernel whose $L^1$-norm grows exponentially fast, which is much worse than the known logarithmic bound. The divergence properties are related to the Mahler measure of certain polynomials and to spectral properties of Ruelle operators.
\end{abstract}

\thanks{This research is partially  supported in part by NSF grant (DMS-1106934 and DMS-1109063) and by a grant from the Simons Foundation \#228539.}
\maketitle \tableofcontents
\section{Introduction}\label{intr}

In \cite{JoPe98} Jorgensen and Pedersen proved that orthogonal Fourier series can be constructed even for some fractal Cantor measures. They considered the Cantor set obtained by dividing the unit interval into four equal pieces and keeping the first and the third piece, and iterating the procedure. The measure $\mu_4$ on this Cantor set is the Hausdorff measure of dimension $\ln 2/\ln 4=1/2$. Define
\begin{equation}
\Lambda:=\left\{\sum_{k=0}^n4^kl_k : l_k\in\{0,1\},n\in\bn \right\}.
\label{eq0.1}
\end{equation}
They proved that the set of exponential functions $\{e^{2\pi i\lambda x}: \lambda\in\Lambda\}$ is an orthonormal basis for $L^2(\mu_4)$. More general examples were later constructed even in higher dimensions \cite{MR1785282,MR1929508,DJ06,DJ07d,MR2338387}.

In \cite{MR2279556}, Strichartz proved the surprising result that the Fourier series for the Jorgensen-Pedersen example have much better convergence properties than their classical counterparts on the unit interval: for example, Fourier series of continuous functions on this Cantor set converge uniformly to the function. To prove this result, Strichartz showed that the corresponding Dirichlet kernel is actually an approximate identity convolution kernel.

In this paper, we will show that this is not always the case, and that for some choices of the digits, the $L^1$-norm of the Dirichlet kernel can grow exponentially fast, a situation much worse even than the known growth in the classical case, which is logarithmic. For example, if for the Jorgensen-Pedersen example, we change the digits in \eqref{eq0.1} from $\{0,1\}$ to $\{0,17\}$ then we still get an orthonormal basis, but the $L^1$-norm of the Dirichlet kernel grows exponentially fast, so the Fourier series do not converge even pointwise.

We will study measures generated by affine iterated function systems (Definition \ref{def1.1}). Lebesgue measure on self-affine tiles appears as a particular example. We show in Section \ref{sec3} that certain rearrangements of the classical Fourier series have the $L^1$-norm of the Dirichlet kernel growing exponentially fast. The divergence rate is related to the Mahler measure of a polynomial associated to this rearrangement.

\begin{definition}\label{def1.3}

We will denote by $e_t$ the exponential function
$$e_t(x)=e^{2\pi it\cdot x},\quad(x,t\in\br).$$

Let $\mu$ be a Borel probability measure on $\br$. We say that $\mu$ is a {\it spectral} measure if there exists a subset $\Lambda$ of $\br$ such that the family $E(\Lambda):=\{e_\lambda : \lambda\in\Lambda\}$ is an orthonormal basis for $L^2(\mu)$. In this case $\Lambda$ is called a {\it spectrum} for the measure $\mu$ and we say that $(\mu,\Lambda)$ is a {\it spectral pair}.
\end{definition}

\begin{definition}\label{def1.1}
We will use the following assumptions throughout the paper.
Let $R$ be a positive integer $R>1$, and let $B$ be a finite subset of $\bz$. In addition, we assume that $0\in B$.

We denote by $N$ the cardinality of $B$. Define the maps
\begin{equation}
\tau_b(x)=R^{-1}(x+b),\quad(x\in\br,b\in B).
\label{eq1.1}
\end{equation}

We call $(\tau_b)_{b\in B}$ the (affine) iterated function system (IFS) associated to $R$ and $B$.

By \cite{Hut81}, there exists a unique compact set $X_B$ called {\it the attractor} of the IFS $(\tau_b)_{b\in B}$ such that
\begin{equation}
X_B=\bigcup_{b\in B}\tau_b(X_B).
\label{eq1.2}
\end{equation}
In our case, it can be written explicitly
\begin{equation}
X_B=\left\{\sum_{k=1}^\infty R^{-k}b_k : b_k\in B\mbox{ for all }k\geq1\right\}.
\label{eq1.3}
\end{equation}

There exists a unique Borel probability measure $\mu=\mu_B$ such that
\begin{equation}
\mu(E)=\frac{1}{N}\sum_{b\in B}\mu(\tau_b^{-1}(E))\mbox{ for all Borel subsets }E\mbox{ of }\br.
\label{eq1.4}
\end{equation}
Equivalently we have the {\it invariance equation}:
\begin{equation}
\int f\,d\mu=\frac1N\sum_{b \in B}\int f\circ \tau_b\,d\mu\mbox{ for all bounded Borel function }f\mbox{ on }\br.
\label{eq1.5}
\end{equation}
The measure $\mu$ is called the {\it invariant measure} of the IFS $(\tau_b)_{b\in B}$. It is supported on $X_B$.

We say that the measure $\mu$ {\it has no overlap} if
\begin{equation}
\mu(\tau_b(X_B)\cap \tau_{b'}(X_B))=0,\mbox{ for all }b\neq b'\in B.
\label{eq1.6}
\end{equation}

\end{definition}

\begin{definition}\label{def1.2}
Let $R$ be given as above.
Let $B$ and $L$ be two subsets of $\bz$ of the same cardinality $N$, with $0\in B$ and $0\in L$. We say that $(B, L)$ form a {\it Hadamard pair} if the following matrix is unitary:
\begin{equation}
\frac{1}{\sqrt N}\left(e^{2\pi i R^{-1}b\cdot l}\right)_{b\in B,l\in L}.
\label{eq1.8}
\end{equation}
We define
\begin{equation}
m_B(x)=\frac{1}{N}\sum_{b\in B}e^{2\pi i b\cdot x},\quad(x\in\br).
\label{eqmb}
\end{equation}

\begin{equation}
m_L(x)=\frac{1}{N}\sum_{l\in L}e^{2\pi i l\cdot x},\quad(x\in\br).
\label{eqml}
\end{equation}

\end{definition}

\begin{definition}
Define the maps $\sigma_l(x)=R^{-1}(x+l)$, $l\in L$. A set of points $C:=\{x_0,x_1,\dots,x_{p-1}\}$ is called an {\it $L$-cycle} if there exist $l_0,\dots,l_{p-1}$ such that $\sigma_{l_0}x_0=x_1,\dots,\sigma_{l_{p-2}}x_{p-2}=x_{p-1}$ and $\sigma_{l_{p-1}}x_{p-1}=x_0$. The $L$-cycle is called {\it extreme} if in addition $|m_B(x_i)|=1$ for all $i\in\{0,\dots,p-1\}$.

\end{definition}

Throughout the paper we will make the following assumptions:
\begin{equation}
\mbox{ $R>1$, $R\in\bz$, $0\in B,L\subset\bz$, $N=\#B=\#L$ and $(B,L)$ form a Hadamard pair.}
\label{eqass}
\end{equation}

\medskip

\begin{proposition}\label{pr1.1}
Assume $(B,L)$ form a Hadamard pair. Then
\begin{enumerate}
	\item The function $m_B$ satisfies the equation
		\begin{equation}
		\sum_{l\in L} |m_B(\sigma_l(x))|^2=1,\quad(x\in\br);
		\label{eqqmfb}
		\end{equation}
	\item The function $m_L$ satisfies the equation
	\begin{equation}
\sum_{b\in B}|m_L(\tau_b(x))|^2=1,\quad(x\in\br)
\label{eqqmfl}
\end{equation}

\end{enumerate}
\end{proposition}

\begin{proof}
By symmetry it is enough to prove (i). We have
$$\sum_{l\in L}|m_B(\sigma_l(x))|^2=\frac{1}{N^2}\sum_{l\in L}\sum_{b,b'\in B}e^{2\pi i R^{-1}(x+l)\cdot (b-b')}=$$$$
\sum_{b,b'\in B}e^{2\pi iR^{-1}x\cdot(b-b')}\frac{1}{N^2}\sum_{l\in L}e^{2\pi i R^{-1}l\cdot (b-b')}=\sum_{b=b'}\frac{1}{N^2}\cdot N=1.$$
\end{proof}

\begin{theorem}\label{thdj}\cite{DJ06}
Let $R,B,L$ be as above, and assume $(B,L)$ form a Hadamard pair.
Let $\Lambda$ be the smallest set with the following properties:
\begin{enumerate}
	\item $R\Lambda+L\subset \Lambda$;
	\item $\Lambda$ contains $-C$ for all extreme $L$-cycles $C$.
\end{enumerate}
Then the measure $\mu$ is spectral with spectrum $\Lambda$.
\end{theorem}

\begin{proposition}\label{pr1.6}
Under the assumptions \eqref{eqass} all extreme $L$-cycles are contained in the lattice $B^\perp$,
\begin{equation}
B^\perp:=\left\{ x\in \br : bx\in\bz\mbox{ for all }b\in B\right\}.
\label{eqbperp}
\end{equation}

Also $B^\perp=\frac{1}{d}\bz$ where $d$ is the greatest common divisor of the elements of $B$.
\end{proposition}

\begin{proof}
 Take $c$ to be a point in an extreme $L$-cycle. Since $|m_B(c)|=1$, we have with the triangle inequality
$$N=\left|\sum_{b\in B}e^{2\pi i bc}\right|\leq N.$$
Therefore, we have equality in the triangle inequality so the terms in the sum differ by a positive multiplicative constant, and since $0\in B$, it follows that $e^{2\pi i bc}=1$ for all $b\in B$. Therefore $bc\in\bz$ for all $b\in B$. The last statement is easy to check since $B$ is contained in $\bz$.

\end{proof}

\begin{lemma}\label{lem1.7}
If $(B,L)$ form a Hadamard pair, then the elements of $L$ are incongruent $\mod RB^\perp$. In particular, they are incongruent $\mod R$. The elements of $B$ are incongruent $\mod R$.
\end{lemma}

\begin{proof}
The Hadamard property implies that for $l\neq l'$ in $L$,
$$\sum_{b\in B}e^{R^{-1}(l-l')\cdot b}=0.$$
This implies that $R^{-1}(l-l')$ cannot be in $RB^\perp$. Since $RB^\perp$ contains $R\bz$, $l$ cannot be congruent to $l'$ $\mod R$. Similarly for $B$.
\end{proof}

\subsection{An algorithm for finding extreme cycles}\label{sec1.1}
Theorem \ref{thdj} shows that in order to find the spectrum of $\mu$, one needs to compute the extreme $L$-cycles. We describe an algorithm for this. First we begin with some properties of extreme $L$-cycles:
\begin{proposition}\label{pr1.7.1}
Let $d$ be the greatest common divisor for $B$.
\begin{enumerate}
	\item Every extreme $L$-cycle point $x_0$ is of the form $\frac{k}{d}$ where $k\in\bz$ and $\frac{\min L}{R-1}\leq \frac kd\leq \frac{\max L}{R-1}$.
	\item For any $k\in\bz$, $|m_B(k/d)|=1$.
	\item For a point $k/d$ with $k\in\bz$, there exists at most one $l\in L$ such that $\frac{1}{R}\left(\frac kd+l\right)\in\frac1d\bz$.
\end{enumerate}

\end{proposition}

\begin{proof}
The first statement in (i) follows from Proposition \ref{pr1.6}. Any $L$-cycle is contained in the attractor
$$X_L=\left\{\sum_{n=1}^\infty R^{-n}l_n : l_n\in L\right\}.$$
Therefore any point in $x$ is between $\min L\cdot\sum_{n=1}^\infty R^{-n}$ and $\max L\cdot\sum_{n=1}^\infty R^{-n}$.

(ii) is trivial.

For (iii), suppose there exist $l\neq l'$ in $L$ such that $R^{-1}(k/d+l)$ and $R^{-1}(k/d+l')$ are both in $\bz/d$. Then $R^{-1}(l-l')\in\bz/d$. But then $R^{-1}(l-l')\cdot b\in\bz$ for all $b\in B$ so $e^{2\pi i R^{-1}(l-l')\cdot b}=1$ for all $b\in B$, and this contradicts the Hadamard property.
\end{proof}

With Proposition \ref{pr1.7.1}, we see that there are only finitely many numbers that we have to check if they are extreme $L$-cycles, namely points of the form $k/d$ between $\min L/(R-1)$ and $\max L/(R-1)$.

Take such a point $k_0/d$. There exists at most one $l\in L$ such that $R^{-1}(k_0/d + l)=k_1/d$ for some $k_1/d$. If there is none, then stop, the point $k/d$ is not an extreme $L$-cycle. If there is one, then $k_1/d$ will be bounded by $\min L/(R-1)$ and $\max L/(R-1)$. Repeat this step with $k_1/d$.

After finitely many steps, we either stop, and then all the point $k_i/d$ obtained in this process are not extreme $L$-cycles, or we return to one of the points $k_i/d$ and then we get an extreme $L$-cycle.

\subsection{Non-overlap and encoding into symbolic spaces}

\begin{theorem}\label{thno}\cite{Alr10}
Let $R$ be an integer, $|R|>1$ and let $\D$ be a set of integers such that no two distinct elements of $\D$ are congruent modulo $R$. Consider the IFS $\tau_d(x)=R^{-1}(x+d)$, $d\in \D$ and let $X(\D)$ be its attractor and $D:=\log_R(\#\D)$. Then the Hausdorff measure of $X(\D)$ satisfies $0<\H^D(X(\D))<\infty$, the invariant measure $\mu_{\D}$ of the IFS $(\tau_d)_{d\in\D}$ is the renormalized Hausdorff measure $\H^D$ restricted to $X(\D)$ and the measure $\mu_{\D}$ has no overlap.
\end{theorem}

\begin{proof}
\newcommand{\inter}{\operatorname*{int}}
Since the elements of $\D$ are incongruent modulo $R$ we can enlarge it to a set $\tilde\D\supset\D$ which is a complete set of representatives for $\bz/R\bz$. We denote by $X(\tilde\D)$ the attractor of the IFS associated to $\tilde\D$.
By \cite[Theorem 1]{Ba91}, the attractor $X(\tilde \D)$ has non-empty interior $\inter(X(\tilde \D))\neq\ty$. Then
$$\cup_{d\in\D}\inter(\tau_d(X(\tilde\D)))\subset\cup_{d\in\tilde\D}\inter(\tau_d(X(\tilde \D)))\subset\inter\left(\cup_{d\in\tilde \D}\tau_d(X(\tilde\D))\right)=\inter(X(\tilde \D)).$$
This means that the Open Set Condition is satisfied for the IFS $(\tau_d)_{d\in\D}$. Using \cite[Theorem 5.3.1 (ii)]{Hut81}, we can conclude that $0<\H^D(X(\D))<\infty$.

For any Borel subset $E$ of $\br$ and $d\in D$, we have
$\H^D(\tau_d^{-1}(E))=\H^D(RE-d)=\H^D(RE)=R^D\H^D(E)=N\H^D(E)$. Similarly $\H^D(\tau_d(E))=\frac{1}{N}\H^D(E)$.

We have
$$\H^D(X(\D))=\H^D\left(\cup_{d\in\D}\tau_d(X(\D))\right)\leq\sum_{d\in\D}\H^D(\tau_d(X(\D)))=\frac{1}{N}\cdot N\H^D(X(\D))=\H^D(X(\D)).$$
Since we must have equality, this implies that $\H^D(\tau_d(X(\D))\cap\tau_{d'}(X(\D)))=0$ for distinct $d,d'\in\D$.

Then for any Borel set $E$:
$$\H^D(E\cap X(\D))=\sum_{d\in\D}\H^D(E\cap \tau_d(X(\D)))=\sum_{d\in\D}\frac{1}{N}\H^D(\tau_d^{-1}(E)\cap X(\D)).$$
But this proves that $\mu_{\D}$ is invariant for the IFS.
\end{proof}

\begin{proposition}\label{pr1.10}
Consider the symbolic space $B^{\bn}$ with the product probability measure $dP$ where each digit in $B$ has probability $1/N$. Define the encoding map:
$\mathcal E: B^{\bn}\rightarrow X_B$,
$$\mathcal E(b_1b_2\dots)=\sum_{n=1}^\infty R^{-n}b_n.$$
Then $\mathcal E$ is onto, it is one-to-one on a set of full measure, and it is measure preserving.

Define the maps for $b\in B$, $\breve b: B^{\bn}\rightarrow B^{\bn}$, $\breve b(b_1b_2\dots)=bb_1b_2\dots$.
Then
$$\mathcal E\circ\breve b=\tau_b\circ \mathcal E.$$
Define the map $\mathcal R:X_B\rightarrow X_B$,
$$\mathcal R\left(\frac{b_1}{R}+\frac{b_2}{R^2}+\dots\right)=\frac{b_2}{R}+\frac{b_3}{R^2}+\dots$$
and define the right shift on $B^{\bn}$,
$$\mathcal S(b_1b_2\dots)=b_2b_3\dots.$$
Then
$$\mathcal E\mathcal S=\mathcal R\mathcal E.$$

\end{proposition}

\begin{proof}
This is standard, see e.g. \cite{Edg08}, and it follows from the non-overlap property proved in Theorem \ref{thno}.
\end{proof}

\section{The Dirichlet kernel}

Next, we define the sets $\Lambda_n$ inductively, starting with extreme $L$-cycles and then scaling by $R$ and adding $L$. The Dirichlet kernel is obtained by summing the exponential functions over the elements in $\Lambda_n$.
\begin{definition}\label{def1.7}
Define the sets $\Lambda_n$ inductively as follows
\begin{equation}
\Lambda_0:=\bigcup\left\{ -C : C\mbox{ is an extreme $L$-cycle }\right\},\quad \Lambda_{n+1}:=R\Lambda_n+L,\quad(n\geq0),
\label{eq1.12}
\end{equation}
We let $\Lambda=\Lambda(L)$ be the set defined in Theorem \ref{thdj}.

Define the {\it Dirichlet kernel}
\begin{equation}
D_n(x):=\sum_{\lambda\in\Lambda_n}e^{2\pi i \lambda x},\quad(x\in\br).
\label{eq1.13}
\end{equation}

Given a function $f\in L^1(\mu)$, we define the partial Fourier series
\begin{equation}
s_n(f;x)=\sum_{\lambda\in\Lambda_n}\left(\int f(y)e^{-2\pi i\lambda y}\,d\mu(y)\right)\cdot e^{2\pi i\lambda x}=\int f(y)D_n(x-y)\,d\mu(y),\quad(x\in\br).
\label{eqsn}
\end{equation}
\end{definition}

\begin{proposition}\label{pr1.8}

The sets $\Lambda_n$ satisfy the following properties:
\begin{enumerate}
	\item $\Lambda_n\subset\Lambda_{n+1}$ for all $n\in\bn$;
	\item $\cup_{n\in\bn}\Lambda_n=\Lambda$;
	\item $\Lambda_n,\Lambda\subset B^\perp$ for all $n\in\bn$.
\end{enumerate}
Let
\begin{equation}
m_c(x):=\sum_{\mbox{extreme $L$-cycle points $c$}}e^{2\pi i(-c)x},\quad(x\in \br).
\label{eq1.14}
\end{equation}

The Dirichlet kernel satisfies the formula
\begin{equation}
D_n(x)=N^nm_c(R^nx)\prod_{k=0}^{n-1}m_L(R^kx),\quad(x,y\in\br).
\label{eq1.15}
\end{equation}

\end{proposition}

\begin{proof}
(i) follows from the definition of $L$-cycles. (ii) follows from the definition of $\Lambda$.  (iii) follows from Proposition \ref{pr1.6} and (ii).

To prove \eqref{eq1.15}, note that
$$D_0(x)=m_c(x),\quad(x\in\br).$$

We claim that
\begin{equation}
D_{n+1}(x)=Nm_L(x)D_n(Rx).
\label{eq1.16}
\end{equation}

We have $\Lambda_{n+1}=R\Lambda_n+L$. By (iii), the elements of $\Lambda_n$ are in $B^\perp$. Using Lemma \ref{lem1.7}, we see that a point $\lambda_{n+1}$ in $\Lambda_{n+1}$ will have a {\it unique} representation of the form $\lambda_{n+1}=R\lambda_n+l$ with $\lambda_n\in\Lambda_n$ and $l\in L$. This implies that
$$D_{n+1}(x)=\sum_{\lambda_n\in\Lambda_n}\sum_{l\in L}e^{2\pi i (R\lambda_n+l)\cdot x}
=\sum_{l\in L}e^{2\pi i l\cdot x}\sum_{\lambda_n\in \Lambda_n}e^{2\pi i R \lambda_n x}=
Nm_L(x)D_n(Rx).$$

Then \eqref{eq1.15} follows from \eqref{eq1.16} by induction.
\end{proof}

\begin{theorem}\label{th2.3}
Let
\begin{equation}
\Delta(Nm_L):=\exp\left(\int \log|Nm_L(x)|\,d\mu(x)\right).
\label{eq2.3.1}
\end{equation}
If $\Delta(Nm_L)>1$ then
\begin{enumerate}
	\item The $L^1$-norm of the Dirichlet kernel grows exponentially fast. More precisely, for any $1<\rho<\Delta(Nm_L)$ there exists a constant $C>0$ such that
	\begin{equation}\label{eq2.3.2}
\|D_n\|_1:=\int |D_n(x)|\,d\mu(x)\geq C\rho^n\mbox{ for $n$ large.}
\end{equation}
\item There exist continuous functions $f$ such that the Fourier series at zero $s_n(f;0)$  is unbounded.
\end{enumerate}
\end{theorem}

\begin{proof}
Consider the map $\mathcal R$ on the attractor $X_B$, defined in Proposition \ref{pr1.10}. Since $m_L$ is $\bz$-periodic
we have that $|m_L(R^kx)|=|m_L(\mathcal R^kx)|$ for all $x\in X_B$ and all $k\in\bn$. Also, we have that $\mathcal R(\tau_b(x))=x$ for all $x\in X_B$ and all $b\in B$.

The invariance equation for $\mu$ shows that for all $f\in C(X_B)$,
$$\int f\circ \mathcal R\,d\mu=\frac1N\sum_{b\in B}\int f\circ\mathcal R\circ \tau_b\,d\mu=\int f\,d\mu.$$

The map $\mathcal R$ is equivalent to the unilateral shift on the symbolic space, as we described in Proposition \ref{pr1.10}. It is well known that the shift is ergodic, therefore the map $\mathcal R$ is ergodic.

By Birkhoff's ergodic theorem, we have that for $\mu$-a.e. $x$ in $X_B$
$$\lim_{n\rightarrow\infty}\frac{1}{n}\sum_{k=0}^{n-1}\log|Nm_L(R^kx)|=\log\Delta(Nm_L).$$

By Egoroff's theorem there exists a subset $A$ of measure $\mu(A)>5/6$ such that the limit above is uniform on $A$.
Take $1<\rho<\Delta(Nm_L)$. There exists $n_\rho$ such that for $n\geq n_\rho$:
$$\frac{1}{n}\sum_{k=0}^{n-1}\log|Nm_L(R^kx)|>\log\rho,\quad(x\in A).$$
Then, for $x\in A$:
$$N^n\prod_{k=0}^{n-1}|m_L(R^kx)|\geq \rho^n.$$

Now take a subset $E$ of $X_B$ and some $\epsilon>0$ such that $|m_c(x)|\geq \epsilon $ for $x\in E$ and $\mu(E)\geq 5/6$. This can be done since $m_c$ is a trigonometric polynomial so it has only finitely many zeros. Then $|m_c|\geq \epsilon \chi_E$.

With these inequalities and Proposition \ref{pr1.8}, we obtain that

$$\int |D_n(x)|\,d\mu(x)\geq \epsilon \rho^n\int \chi_A(x)\chi_E(R^nx)\,d\mu(x)=\epsilon\rho^n\int \chi_A(x)\chi_{\mathcal R^{-n}E}(x)\,d\mu(x).$$

But since the measure $\mu$ is invariant for $\mathcal R$ it follows that $\mu(\mathcal R^{-n}E)=\mu(E)$. Then
$$\mu(A\cap \mathcal R^{-n}E)=\mu(A)+\mu(\mathcal R^{-n}E)-\mu(A\cup \mathcal R^{-n}E)\geq 5/6+5/6-1=2/3.$$

Finally, this means that for $n\geq n_\rho$:
$$\int|D_n(x)|\,d\mu(x)\geq \frac23\epsilon \rho^n.$$

To prove (ii) we use the same argument as in the classical case, see e.g. \cite[Chapter 5]{Rud87}.

Define the linear functionals $\varphi_n:C(X_B)\rightarrow\bc$, $\varphi_n(x)=s_n(f;0)$.
Then
$$\varphi_n(f)=\int f(x)\cj{D_n(x)}\,d\mu(x).$$
We then see that $\|\varphi_n\|\leq \|D_n\|_1$. On the other hand put $g(x)=|D_n(x)|/\cj{D_n(x)}$ if $|D_n(x)|\neq 0$, $g(x)=1$ otherwise (Note that $D_n(x)$ has finitely many zeros). There exist $f_j\in C(X_B)$ such that $|f_j|\leq 1$ such that $f_j$ converges to $g$ pointwise. By the dominated convergence theorem
$$\lim_{j\rightarrow\infty}\varphi_n(f_j)=\int |D_n(x)|\,d\mu(x)=\|D_n\|_1.$$

This shows that
\begin{equation}
\|\varphi_n\|=\|D_n\|_1.
\label{eq2.3.3}
\end{equation}
Then the hypothesis implies that $\lim_n\|\varphi_n\|=\infty$. By the Banach-Steinhaus uniform boundedness principle, there exists a function $f\in C(X_B)$ such that $s_n(f;0)=\varphi_n(f)$ is unbounded.
\end{proof}

Given  the spectrum set $\Lambda:=\cup_{n\ge 0}\Lambda_n$  associated with two subsets $B$ and $L$ having the same cardinality $R\ge 2$ (c.f. Examples \ref{ex3.7} and \ref{ex3.8}),
  there is a continuous function $f$ with spectrum $\Lambda$ by Theorem \ref{th2.3} such that 
 its Fourier series diverges at the origin, provided that $\Delta(Nm_L)$ in \eqref{eq2.3.1} is strictly larger than one.
The  pointwise convergence of Fourier series with given spectrum, such as polynomial spectrum, is related to many mathematical branches, 
such as number theory, and its investigation dates back to Gauss and Weyl \cite{ArkOs89, Vino85}.

\begin{proposition}\label{pr2.4}
We have the following bound on the number $\Delta(Nm_L)$ in Theorem \ref{th2.3}:
\begin{equation}
\Delta(Nm_L)<\sqrt{N}.
\label{eq2.4.1}
\end{equation}
\end{proposition}

\begin{proof}
We have, by the invariance equation and Proposition \ref{pr1.1}:
$$\int |m_L|^2\,d\mu=\frac1N\sum_{b\in B}\int |m_L(\tau_bx)|^2\,d\mu=\frac1N.$$
Then, using Jensen's inequality:
$$\int \log|Nm_L|\,d\mu=\frac12\int\log|Nm_L|^2\,d\mu< \frac12\log\left(\int|Nm_L|^2\,d\mu\right)=\frac12\log(N)=\log\sqrt{N}.$$
The inequality is strict because $|Nm_L|$ is not constant a.e.
\end{proof}

\begin{example}\label{ex2.5}
Consider the Jorgensen-Pedersen example \cite{JoPe98}, $R=4$, $B=\{0,2\}$. One can pick $L=\{0,p\}$ for any odd number $p$. In \cite{JoPe98}, Jorgensen and Pederesen picked $L=\{0,1\}$. In \cite{MR2279556} Strichartz proved that for this choice $L=\{0,1\}$, the Dirichlet kernel is an approximate identity convolution kernel so Fourier series of continuous functions converge uniformly to the function, and the $L^1$-norm of the Dirichlet kernel is bounded.

 We have for $L=\{0,p\}$:
 $$2|m_L(x)|=|1+e^{2\pi i px}|=2|\cos(p\pi x)|.$$

We list below the numerical approximation for $\Delta(2m_L)$, and the non-trivial extreme $L$-cycles ($\{0\}$ is the trivial extreme $L$-cycle).
\begin{center}
\begin{tabular}{|c|c|c|}
\hline
$L$ & $\Delta(2m_L)$ & Non-trivial extreme $L$-cycles\\
\hline
$\{0,1\}$& 0.645591 & none\\ \hline
$\{0,3\}$&0.572511 & $\{1\}$\\ \hline
$\{0,5\}$&0.977119 & none\\ \hline
$\{0,7\}$ &0.92037 & none\\ \hline
$\{0,9\}$ &0.799365 &$\{3\}$\\ \hline
$\{0,11\}$ &0.876405 & none \\ \hline
$\{0,13\}$ &0.857008 &  none \\ \hline
$\{0,15\}$ &0.596433 & $\{1,4\}$, $\{5\}$\\ \hline
$\{0,17\}$ &1.04887 & none\\ \hline
$\{0,19\}$ &0.879154 & none\\ \hline
$\{0,21\}$ &0.967384 &$\{7\}$\\ \hline
$\{0,23\}$ &1.0163 & none\\ \hline
$\{0,25\}$ &0.943818 & none\\ \hline
$\{0,27\}$ &0.921974 & $\{9\}$\\ \hline
$\{0,29\}$ &1.00966 &  none\\ \hline
\end{tabular}
\end{center}

Note that for $p=17,23$ and $29$ we get $\Delta(2m_L)>1$, which implies that the corresponding Dirichlet kernels grow exponentially fast in $L^1$-norm and the Fourier series diverge for some continuous functions.

The numerical approximations in the table above were computed using an algorithm based on Elton's theorem for iterated function systems \cite{Elt87}. We use Elton's theorem because we integrate with respect to the fractal measure $\mu$, not the Lebesgue measure. Elton's ergodic theorem asserts that for a contractive iterated function system $\{\tau_i\}_{i=1}^n$ with invariant measure $\mu$ and for every continuous function $f$ on the attractor, the averages
$$\frac{1}{N}\sum_{j=0}^{N-1}f\circ\tau_{i_{j}}\circ\dots\circ\tau_{i_0}x$$
converge to $\int f\,d\mu$ for almost every random choice of digits $i_0,i_1,\dots,$ and every $x$ in the attractor. 

Another way to approximate the integral $\int f\,d\mu$ is by $\frac{1}{|B|^n}\sum_{b_1,\dots,b_n\in B}f\circ\tau_{b_n}\circ\dots\circ\tau_{b_1}x$, for some $x\in X_B$.
%

The extreme $L$-cycles can be computed using the algorithm in Section \ref{sec1.1}.
\end{example}

\subsection{Ruelle operators}

\begin{definition}\label{def3.1}
For functions $f$ defined on $X_B$, we define the {\it Ruelle operator} $R_L$ associated to the function $|m_L|$:
\begin{equation}
(R_Lf)(x)=\sum_{b\in B}|m_L(\tau_bx)|f(\tau_bx),\quad(x\in X_B).
\label{eq3.1.1}
\end{equation}

\end{definition}

\begin{theorem}\label{th3.1}
Define
\begin{equation}
\mathfrak D:=\lim_{n\rightarrow\infty}(\sup_{x\in X_B} (R_L^n1)(x))^{\frac1n}.
\label{eq3.1.2}
\end{equation}
If $\mathfrak D>1$ then $\|D_n\|_1$ grows exponentially fast. More precisely, for each $1<\rho\leq \mathfrak D$ there exists a constant $C>0$ such that
\begin{equation}
\|D_n\|_1\geq C\rho^n,\mbox{ for $n$ large.}
\label{eq3.1.3}
\end{equation}
In particular, there exist continuous functions $f\in C(X_B)$ such that $s_n(f;0)$ is unbounded.

Also, the converse is true: if there exists $\rho>1$ and $C>0$ such that \eqref{eq3.1.3} holds, then $\mathfrak D>1$.
\end{theorem}

We will need some lemmas.

\begin{lemma}\label{lem3.3}
For every $n\in\bn$
\begin{equation}
\|D_n\|_1=\int |m_c(x)|(R_L^n1)(x)\,d\mu(x).
\label{eq3.3.1}
\end{equation}
\end{lemma}

\begin{proof}
We use \eqref{eq1.15} and the invariance equation:
$$\int |D_n(x)|\,d\mu(x)=\int\sum_{b_0,\dots,b_{n-1}\in B}|m_c(R^n\tau_{b_{n-1}}\dots\tau_{b_0}x)|\prod_{k=0}|m_L(R^k\tau_{b_{n-1}}\dots\tau_{b_0}x)|\,d\mu(x)$$
$$=\int |m_c(x)|(R_L^n1)(x)\,d\mu(x).$$

\end{proof}

\begin{theorem}\label{th3.4}{\bf [Ruelle-Perron-Frobenius Theorem]}\cite{Bal00}
There exists a function $h\in L^\infty(\mu)$, $h\geq 0$, $h\neq 0$ such that $R_Lh=\mathfrak D h$.
\end{theorem}

\begin{remark}\label{rem3.5r}
Ruelle's theorem is formulated for the symbolic space and actually gives a continuous $h$. We can transfer the result to the attractor $X_B$ using the encoding in Proposition \ref{pr1.10}, but since there might be some points of overlap (of measure zero), the continuity of $h$ might be lost. However this happens only on a set of measure zero so the resulting function $h$ is still bounded.
\end{remark}

\begin{proof}[Proof of Theorem \ref{th3.1}]
Take $h$ as in Theorem \ref{th3.4}. Then $1\geq Ch$ for some $C>0$. Using the results above we have
$$\| D_n\|_1=\int |m_c(x)|R_L^n1(x)\,d\mu(x)\geq \int |m_c|CR_L^nh\,d\mu=C\int|m_c|\mathfrak D^n\cdot h\,d\mu$$$$=C\mathfrak D^n\int |m_c|h\,d\mu=C'\mathfrak D^n.$$
$C'$ is positive because $h\geq0, h\neq 0$ and $|m_c|$ is zero only at finitely many points.

The conclusion that $s_n(f,0)$ is unbounded follows as in the proof of Theorem \ref{th2.3}.

For the converse, \eqref{eq3.1.3} and \eqref{eq3.3.1} implies
$$\sup R_L^n1\cdot\int |m_c|\,d\mu\geq C\rho^n\mbox{ for $n$ large,}$$
and this shows that $\mathfrak D\geq\rho>1$.

\end{proof}

\begin{proposition}\label{pr3.2}
The operator $R_L$ is positive, i.e, $R_Lf\geq 0$ if $f\geq 0$. Also
\begin{equation}
R_L1\geq 1
\label{eq3.2.1}
\end{equation}
Consequently $\mathfrak D\geq 1$ and the sequence $(\|D_n\|_1)_n$ is increasing.
\end{proposition}

\begin{proof}
First statement is trivial. For \eqref{eq3.2.1}, using Proposition \ref{pr1.1}, we have that $|m_L|\leq 1$ and
$$R_L1(x)=\sum_{b\in B}|m_L(\tau_bx)|\geq\sum_{b\in B}|m_L(\tau_bx)|^2=1.$$
By induction $R_L^{n+1}1\geq R_L^n1\geq 1$ so $\mathfrak D\geq 1$ and, with \eqref{eq3.3.1}, the sequence $\|D_n\|_1$ is increasing.
\end{proof}

\begin{proposition}\label{pr3.3}
Suppose there exist $h\in L^\infty(\mu)$, $h\geq c>0$, $\mu$-a.e., such that $R_Lh=h$. Then $\mathfrak D=1$ and $\|D_n\|_1$ is bounded.
\end{proposition}

\begin{proof}
We have $R_L^n1\leq \frac1c R_L^nh=\frac1ch$. The rest follows from the definition of $\mathfrak D$ and from Lemma \ref{lem3.3}.
\end{proof}

\begin{proposition}\label{pr3.4}
Let $h\geq 0$, $h\neq 0$ be a continuous fixed point of the transfer operator $R_Lh=h$, $h\in C(X_B)$. Then the zeroes of $h$ are contained in the extreme $B$-cycles, i.e., if $h(z_0)=0$ then $z_0$ is in a cycle $C$ for the IFS $(\tau_b)_{b\in B}$ and $|m_L(z)|=1$ for all $z\in C$.
\end{proposition}

\begin{proof}
We call a point $z\in\br$ periodic if $R^pz\equiv z\mod\bz$. Let's prove first that $h(z_0)=0$ implies that $z_0$ is periodic. Suppose not. Then we claim that the points $\tau_{b_{n-1}}\dots\tau_{b_0}z_0\neq \tau_{b_{m-1}'}\dots\tau_{b_0'}z_0$ if $b_0\dots b_{n-1}\neq {b_0'}\dots b_{m-1}'$. If not, then, assume $n\geq m$, and we have
$$z_0+b_0+\dots+R^{n-1}b_{n-1}=R^{n-m}z_0+R^{n-m}b_0'+\dots+R^{n-1}b_{m-1}'.$$
If $n>m$ then this implies that $z_0\equiv R^{n-m}z_0\mod\bz$ so $z_0$ is periodic. If $n=m$, this is impossible since the elements of $B$ are not congruent $\mod R\bz$.

Since $R_Lh=h$ we have
$$0=h(z_0)=\sum_{b\in B}|m_L(\tau_bz_0)|h(\tau_bz_0).$$
Then for all $b\in B$, we have either $m_L(\tau_bz_0)=0$ or $h(\tau_bz_0)=0$. We cannot have $m_L(\tau_bz_0)=0$ for all $b\in B$, because of \eqref{eqqmfl}. Therefore there exists $b_0\in B$ such that $z_1=\tau_{b_0}z_0$ is a zero for $h$. By induction, we can find $b_1,\dots,b_n,\dots$ in $B$ such that $z_n=\tau_{b_{n-1}}\dots\tau_{b_0}z_0$ is a zero for $h$.

If $z_0$ is not periodic then we proved above that the points $\tau_{a_n}\dots\tau_{a_0}z_0$ are distinct for all words $a_0\dots a_n$ with letters in $B$. Since $m_L$ has finitely many zeros, it follows that for $n$ large enough $m_L(\tau_{a_n}\dots\tau_{a_0}z_0)\neq 0$. Fix $n$ large, the previous argument shows that $h(\tau_{a_k}\dots\tau_{a_0}z_n)=0$ for all $k$ and all $a_0,\dots,a_k\in B$. But then $h$ is zero on a dense subset of $X_B$ so $h=0$, since $h$ is continuous.

This contradiction shows that $z_0$ is periodic, and the same argument shows that $z_n$ is periodic for all $n$.

We have $z_1=\tau_{b_0}z_0$ and $z_1$ is periodic. If we take $b_0'\in B$, $b_0'\neq b_0$ then $\tau_{b_0'}z_0$ cannot be periodic (since $b_0,b_0'$ are incongruent $\mod R\bz$). Therefore, using again the same argument as before we cannot have $h(\tau_{b_0'}z_0)=0$ so we must have $m_L(\tau_{b_0'}z_0)=0$. Then, using \eqref{eqqmfl}, we must have $|m_L(\tau_{b_0}z_0)|=1$. By induction, we obtain $|m_L(z_n)|=1$ for all $n$, and since each $z_n$ is periodic, the points $z_n$ are on a cycle for $(\tau_b)_{b\in B}$.
\end{proof}

\section{Lebesgue measure}\label{sec3}

In this section we will make the following assumptions:

\begin{equation}
\mbox{ The sets $B$ and $L$ are complete sets of representatives for $\bz/R\bz$}
\label{eq4.1}
\end{equation}
and
\begin{equation}
\mbox{The greatest common divisor of $B$ is 1.}
\label{eq4.2}
\end{equation}
We will apply Theorems \ref{th2.3} and \ref{th3.1} to the divergence of certain rearranged Fourier series of a continuous function, see
Subsection \ref{scambled.subsection}.

We denote by $\mathcal L$ the Lebesgue measure on $\br$.
The following theorem is known:

\begin{theorem}\label{th4.2}
Under the assumptions \eqref{eq4.1} and \eqref{eq4.2}, the invariant measure $\mu$ is the Lebesgue measure $\mathcal L$ restricted to $X_B$,  and the Lebesgue measure of $X_B$ is $\mathcal L(X_B)=1$. Moreover
$X_B$ tiles $\br$ by translations with $\bz$, and $\mu$ has spectrum $\bz$.
\end{theorem}

\begin{proof}
It is known, see e.g. \cite{Wan99}, that $\mu$ is Lebesgue measure on $X_B$ renormalized by $\mathcal L(X_B)$, and it tiles $\br$ by a sublattice of $\bz$. We just have to prove that this sublattice is $\bz$.
Suppose $X_B$ tiles by $d\bz$. Then $\mu$ has spectrum the dual lattice $\frac1d\bz$. We analyze the extreme $L$-cycles. By Proposition \ref{pr1.6}, these are contained in $B^\perp=\bz$ since the greatest common divisor of $B$ is 1. Thus all the extreme $L$-cycles are contained in $\bz$. So the spectrum is contained in $\bz$, so it has to be $\bz$.
\end{proof}

\begin{definition}\label{def4.3}
We define the Laurent polynomial
$$p_L(z):=\sum_{l\in L}z^l.$$
For a Laurent polynomial $p(z)$, we define the {\it Mahler measure} (see \cite{BoEr95})
$$\Delta(p)=\exp\left(\int_0^1\log|p(e^{2\pi ix})|\,dx\right).$$
\end{definition}

\begin{remark}\label{rem3.5}
If $p(z)=a(z-z_1)\dots(z-z_n)$ is a polynomial with $z_1,\dots, z_n$ as its roots, then the Mahler measure is (see \cite{BoEr95}):

\begin{equation}
\Delta(p)=a\prod_{k=1}^n\max\{1,|z_k|\}.
\label{eqmm}
\end{equation}
\end{remark}

\begin{theorem}\label{th4.4}
Suppose  that \eqref{eq4.1} and \eqref{eq4.2} are satisfied. If the Mahler measure $\Delta(p_L)>1$, then for every $1<\rho<\Delta(p_L)$, there exists a constant $C>0$ such that
$$\int |D_n(x)|\,d\mu(x)\geq C\rho^n\ \mbox{ for large $n$}.$$
\end{theorem}

\begin{proof}
We use a Lemma:
\begin{lemma}\label{lem4.5}
If the function $f$ is $\bz$-periodic, then
$$\int f\,d\mu=\int_0^1 f(x)\,dx.$$
\end{lemma}

\begin{proof}
By Theorem \ref{th4.2}, $\mu$ is the Lebesgue measure on $X_B$, $\mathcal L(X_B)=1$ and $X_B$ tiles $\br$ by $\bz$. This implies that $X_B$ is translation congruent to $[0,1]$, i.e., there exists a partition $(A_k)_{k\in\bz}$ of $X_B$ such that $(A_k-k)_{k\in \bz}$ is a partition of $[0,1]$. Then the lemma follows.
\end{proof}

Note that $Nm_L(x)=p_L(e^{2\pi i x})$. Theorem \ref{th4.4} follows from Theorem \ref{th2.3}.
\end{proof}

\subsection{Scrambled Fourier series}\label{scambled.subsection}
We consider now the case $B=\{0,\dots,R-1\}$. In this case the measure $\mu$ is the Lebesgue measure on the unit interval. We take $L$ to be a complete set of representatives modulo $R$. The sets $\Lambda_n$ will give a certain ``scrambling'', i.e., a rearrangement of the integers with algebraic structure, cf. Example \ref{ex3.7} and \ref{ex3.8}.

\begin{definition}\label{def4.6}
Let $R>1$ be an integer and let $L$ be a complete set of representatives for $\bz/ R\bz$, $0\in L$. We define the following subsets of $\bz$:
$$\Lambda_0:=-\bigcup\left\{C : C\mbox{ is an extreme $L$-cycle}\right\},$$
$$\Lambda_{n+1}=R\Lambda_n+L,\mbox{ for all $n\geq0$}.$$
Note that we do have (see Proposition \ref{pr1.8}):
$$\Lambda_n\subset\Lambda_{n+1}\mbox{ and }\bigcup_{n\in\bn}\Lambda_n=\bz.$$
We define the  partial sum of {\it scrambled Fourier series}, by summing over the sets $\Lambda_n$:
\begin{equation}\label{scambledfourier.def}
s_n(f,x)=\sum_{\lambda\in\Lambda_n}\left(\int_0^1f(x)e^{-2\pi i\lambda y}\,dy\right)\cdot e^{2\pi i\lambda x},\quad(x\in\br),\end{equation}
and the {\it scrambled Dirichlet kernel}
\begin{equation}\label{scambleddirichlet.def} D_n(x)=\sum_{\lambda\in\Lambda_n}e^{2\pi i \lambda x},\quad(x\in\br,n\in\bn).\end{equation}
\end{definition}

\begin{example}\label{ex3.7}
For the special case that $L=\{0,1,...,R-1\}$ with $R\ge 2$, the  partial sum $s_n(f,x)$ of scrambled Fourier series in \eqref{scambledfourier.def}
and the  scrambled Dirichlet kernel $D_n(x)$ in \eqref{scambleddirichlet.def}
become the  traditional dyadic partial sum of Fourier series
$$\sum_{\lambda=-R^n}^{R^n-1} \left(\int_0^1f(x)e^{-2\pi i\lambda y}\,dy\right)\ e^{2\pi i\lambda x}$$
and the Dirichlet kernel $$\sum_{\lambda=-R^n}^{R^n-1}e^{2\pi i \lambda x}=e^{-\pi ix} \frac{\sin (2\pi R^nx)}{\sin \pi x}$$ respectively.
The reasons are that the extreme cycles can only be 0 and 1 and both of them are indeed 1-cycles as $$0=0/R\mbox{ and } 1=(1+(R-1))/R,$$
and that
 $\Lambda_n, n\ge 0$, in Definition \ref{def4.6} are as follows:  $\Lambda_0=\{-1,0\}$,
$$\Lambda_1= R \Lambda_0+L=\{-R,-R+1,...,0, ...,R-1\},$$ and
$$\Lambda_k=R\Lambda_{k-1}+L=\{-R^k,-R^k+1,...,0,...,R^k-1\}$$ for $k\ge 2$ by induction.
\end{example}

Taking $B:=\{0,\dots,R-1\}$, we have the following corollary from Theorems \ref{th4.4} and \ref{th2.3}.

\begin{corollary}\label{cor4.7}
If the Mahler measure $\Delta(p_L)>1$, then
 \begin{itemize}
 \item [{(i)}] The $L^1$-norm of the scrambled Dirichlet kernel grows exponentially fast, more precisely for any $1<\rho<\Delta(p_L)$ there exists $C>0$ such that
$$\int_0^1|D_n(x)|\,dx\geq C\rho^n\mbox{ for $n$ large.}$$

\item [{(ii)}] There exist a continuous function $f$ such that the scambled Fourier series at zero $s_n(f; 0)$  in \eqref{scambledfourier.def} is unbounded.
\end{itemize}
\end{corollary}

\begin{remark} Associated with a complete set $L$ of representatives for $\bz/ R\bz$, the sets $\Lambda_n, n\ge 0$, in Definition \ref{def4.6} 
 give a rearrangement of the integers with certain algebraic structure, cf. Example \ref{ex3.7} and \ref{ex3.8}.
 By Corollary \ref{cor4.7}, if
 the Mahler measure $\Delta(p_L)$ of the complete set $L$ is strictly  larger than one, then there is a continuous function and a rearrangement  associated with the complete set $L$ such that the rearrangement of its Fourier series  diverges at the origin.
 The pointwise convergence and divergence of   Fourier series and its rearrangements is one of fundamental  problems in Fourier analysis, see
 \cite{aubry06, Carleson66,  Fefferman73, KK66, Kolmogorov23, KonyaginT03, Konyagin08, Korner96, Korner99, Hunt67} and 
 the recent survey paper \cite{Konyagin06}. We remark that  Fourier series  in Corollary \ref{cor4.7}
 is rearranged  according to  frequencies, while in \cite{KonyaginT03, Korner96, Korner99}  Fourier series are reconstructed
  according to the amplitudes, a conventional greedy algorithm. One of the main advantages of the rearrangements described in our paper is that they still have a certain affine structure; they are obtained by applying simple dilations and translations to the original set $\Lambda_0$. 
\end{remark}

\begin{example}\label{ex3.8}
Take $R=3$, $B=\{0,1,2\}$, and $L=\{0,1,5\}$ (a complete set of representatives modulo $R=3$).
We list the first few sets $\Lambda_n$ explicitly. The extreme $L$-cycles are $\{0\}$ and $\{1,2\}$, therefore
$$\Lambda_0=\{0,-1,-2\}.$$
Then, we use \eqref{eq1.12} to obtain the sets $\Lambda_n$ inductively:
$$\Lambda_1=\{-6, -5, -3, -2, -1, 0, 1, 2, 5\},$$
\begin{eqnarray*}
\Lambda_2&=\{-18, -17, -15, -14, -13, -10, -9, -8, -6, -5, -4, -3, -2, -1, 0, 1,\\
&2, 3, 4, 5, 6, 7, 8, 11, 15, 16, 20\},
\end{eqnarray*}
\begin{eqnarray*}
\Lambda_3&=&\{-54, -53, -51, -50, -49, -46, -45, -44, -42, -41, -40, -39, -38, -37,
-34, -30, -29,\\ & &-27, -26, -25, -24, -23, -22, -19, -18, -17, -15, -14,
-13, -12, -11, -10, -9, -8, \\& &-7, -6,-5, -4, -3, -2, -1, 0, 1, 2, 3,
4, 5, 6, 7, 8, 9, 10, 11, 12, 13, 14,15, 16, 17, 18, 19,\\& &  20, 21, 22,
23, 24, 25, 26, 29, 33, 34, 38, 45, 46, 48, 49, 50, 53, 60, 61, 65\}.
\end{eqnarray*}
Note that there are many gaps in the sets $\Lambda_n$. We have
$$\min\Lambda_n=-2\cdot 3^n,\quad\max\Lambda_n=\sum_{k=0}^{n-1}5\cdot 3^k=\frac{5}{2}(3^n-1).$$
The set $\Lambda_n$ has $3^{n+1}$ elements. Thus the set $\Lambda_n$ contains only a fraction of
$$\frac{3^{n+1}}{\frac{5}{2}(3^n-1)-(-2\cdot 3^n)+1}\approx\frac23$$
from the integers in the range $[\min\Lambda_n,\max\Lambda_n]$.
Notice that the polynomial
$$p_L(z)=1+z+z^5$$
has a root $0.877439 - 0.744862 i$ which has absolute value $1.15096$. Therefore its Mahler measure is strictly bigger than 1, see Remark \ref{rem3.5}. By Corollary \ref{cor4.7}, the $L^1$-norm of the Dirichlet kernel grows exponentially fast, and there exists a continuous  periodic function $f$
such that the scambled Fourier series at zero $s_n(f; 0)$  is unbounded.
\end{example}

\noindent {\bf Example \ref{ex3.7} (revisited)}\
Take $R\ge 2$, $B:=\{0,\dots,R-1\}$ and $L=\{0,\dots,R-1\}$. Then 
$$p_L(z)=1+z+\dots+z^{R-1},\quad Nm_L(x)=p_L(e^{2\pi i x}),\quad(x\in\br,z\in\bc).$$
The roots of $p_L$ are the non-trivial roots of unity, therefore the Mahler measure of $p_L$ (by Remark \ref{rem3.5}) is 1, and it is equal to $\Delta(Nm_L)$.
%
The Dirichlet kernel
$$D_n(x)=\sum_{k=-R^n}^{R^n-1}e^{2\pi ikx}=e^{-\pi ix} \frac{\sin (2\pi R^nx)}{\sin \pi x},$$ 
satisfies
$$\| D_n\|_1\approx\log R^n=n \log R$$
\cite{Zyg02}.
With Theorem \ref{th3.1}, we obtain that the number defined in \eqref{eq3.1.2} in relation to the Ruelle operator $R_L$,
$\mathfrak D=1$. By Ruelle's theorem \ref{th3.4}, there exists a fixed point $h\geq0$, $h\in L^\infty[0,1]$ such that
$R_Lh=h$. However such an $h$ cannot be bounded away from zero, because with Proposition \ref{pr3.3} that would imply that the $L^1$-norm of the Dirichlet kernel is bounded.
Thus for this particular choice of $R,B,L$, the Ruelle operator $R_L$ has no bounded fixed points which are bounded away from zero.

\bigskip

We end this paper with two questions:
\begin{question}
Fix the integer $R>1$. What is
$$\mathfrak d_R:=\sup\left\{\Delta(p_L) : L\mbox{ is a complete set of representatives modulo $R$}\right\}?$$
\end{question}

If $\mathfrak d_R$ is big, this means, by Corollary \ref{cor4.7}, that we can find scramblings of the Fourier series that have an exponential growth of the $L^1$-norm of the Dirichlet kernel, with as bad a rate as $\mathfrak d_R$.

\begin{remark}
We can get an easy upper bound for $\mathfrak d_R$. We have
\begin{equation}
\mathfrak d_R\leq \sqrt{R}.
\label{eqrdr}
\end{equation}
Indeed, take $L$ a complete set of representatives modulo $R$. Using Jensen's inequality we have

$$\int_0^1\log|p_L(e^{2\pi ix})|\,dx=\frac12\int_0^1\log|p_L(e^{2\pi ix})|^2\,dx\leq \frac12\log\left(\int_0^1|p_L(e^{2\pi ix})|^2\,dx\right).$$
But
$$\int_0^1|p_L(e^{2\pi ix})|^2\,dx=\int_0^1\sum_{l,l'\in L}e^{2\pi i(l-l')x}\,dx=R.$$

This implies that
$$\Delta(p_L)=\exp\left(\int_0^1\log|p_L(e^{2\pi i x})|\,dx\right)\leq \exp\left(\frac12\log R\right)=\sqrt{R}.$$
Therefore $\mathfrak d_R\leq \sqrt{R}.$
\end{remark}
\begin{question}
Can one find $R>1$ and $L$ a complete set of representatives modulo $R$, such that the Ruelle operator $R_L$ has a fixed point $h$ with $0<c\leq h\leq C$, $R_Lh=h$?
\end{question}

If this is true, then by Proposition \ref{pr3.3}, there we can find a scrambling of the Fourier series that has a Dirichlet kernel bounded in $L^1$-norm.

\begin{acknowledgements}
We would like to thank Professors Palle Jorgensen, Keri Kornelson and Gabriel Picioroaga for helpful discussions and ideas. We also thank the anonymous referee for his/her suggestions and for pointing out several important references. 
\end{acknowledgements}
\bibliographystyle{alpha}
\bibliography{divergence_revised}

\begin{thebibliography}{Wan99}

\bibitem[Alr10]{Alr10}
Bengt Alrud.
\newblock Fractal spectral measures in two dimensions.
\newblock {\em PhD thesis, University of Central Florida}, 2010.

\bibitem[AO89]{ArkOs89}
G.~I. Arkhipov and K.~I. Oskolkov.
\newblock On a special trigonometric series and its applications.
\newblock {\em Math. USSR Sbornik}, 62:145--155, 1989.

\bibitem[Aub06]{aubry06}
Jean-Marie Aubry.
\newblock On the rate of pointwise divergence of fourier and wavelet series in
  $l^p$.
\newblock {\em J. Approx. Theory}, 138(1):97--111, 2006.

\bibitem[Bal00]{Bal00}
Viviane Baladi.
\newblock {\em Positive transfer operators and decay of correlations},
  volume~16 of {\em Advanced Series in Nonlinear Dynamics}.
\newblock World Scientific Publishing Co. Inc., River Edge, NJ, 2000.

\bibitem[Ban91]{Ba91}
Christoph Bandt.
\newblock Self-similar sets. {V}. {I}nteger matrices and fractal tilings of
  {${\bf R}^n$}.
\newblock {\em Proc. Amer. Math. Soc.}, 112(2):549--562, 1991.

\bibitem[BE95]{BoEr95}
Peter Borwein and Tam{\'a}s Erd{\'e}lyi.
\newblock {\em Polynomials and polynomial inequalities}, volume 161 of {\em
  Graduate Texts in Mathematics}.
\newblock Springer-Verlag, New York, 1995.

\bibitem[Car66]{Carleson66}
Lennart Carleson.
\newblock On convergence and growth of partial sums of {F}ourier series.
\newblock {\em Acta Math.}, 116:135--157, 1966.

\bibitem[DJ06]{DJ06}
Dorin~Ervin Dutkay and Palle E.~T. Jorgensen.
\newblock Iterated function systems, {R}uelle operators, and invariant
  projective measures.
\newblock {\em Math. Comp.}, 75(256):1931--1970 (electronic), 2006.

\bibitem[DJ07]{DJ07d}
Dorin~Ervin Dutkay and Palle E.~T. Jorgensen.
\newblock Fourier frequencies in affine iterated function systems.
\newblock {\em J. Funct. Anal.}, 247(1):110--137, 2007.

\bibitem[Edg08]{Edg08}
Gerald Edgar.
\newblock {\em Measure, topology, and fractal geometry}.
\newblock Undergraduate Texts in Mathematics. Springer, New York, second
  edition, 2008.

\bibitem[Elt87]{Elt87}
John~H. Elton.
\newblock An ergodic theorem for iterated maps.
\newblock {\em Ergodic Theory Dynam. Systems}, 7(4):481--488, 1987.

\bibitem[Fef73]{Fefferman73}
Charles~Louis Fefferman.
\newblock Pointwise convergence of {F}ourier series.
\newblock {\em Ann. of Math.}, 98:551--571, 1973.

\bibitem[Hun67]{Hunt67}
Richard~A. Hunt.
\newblock On the convergence of {F}ourier series.
\newblock In {\em Orthogonal Expansions and their Continuous Analogues (Proc.
  Conf., Edwardsville, Ill., 1967)}, pages 235--255. Southern Illinois Univ.
  Press, Carbondale, Ill., 1967.

\bibitem[Hut81]{Hut81}
John~E. Hutchinson.
\newblock Fractals and self-similarity.
\newblock {\em Indiana Univ. Math. J.}, 30(5):713--747, 1981.

\bibitem[JKS07]{MR2338387}
Palle E.~T. Jorgensen, Keri~A. Kornelson, and Karen~L. Shuman.
\newblock Affine systems: asymptotics at infinity for fractal measures.
\newblock {\em Acta Appl. Math.}, 98(3):181--222, 2007.

\bibitem[JP98]{JoPe98}
Palle E.~T. Jorgensen and Steen Pedersen.
\newblock Dense analytic subspaces in fractal {$L\sp 2$}-spaces.
\newblock {\em J. Anal. Math.}, 75:185--228, 1998.

\bibitem[KK66]{KK66}
Jean-Pierre Kahane and Yitzhak Katznelson.
\newblock Sur les ensembles de divergence des s\'eries trigonom\'etriques.
\newblock {\em Studia Math.}, 26:305--306, 1966.

\bibitem[Kol23]{Kolmogorov23}
Andrey Kolmogorov.
\newblock Une s\'erie de {F}ourier–{L}ebesgue divergente presque partout.
\newblock {\em Fundamenta Math.}, 4:324--328, 1923.

\bibitem[Kon06]{Konyagin06}
Sergey~V. Konyagin.
\newblock Almost everywhere convergence and divergence of {F}ourier series.
\newblock In {\em Proceeding of the International Congress of Mathematics,
  Madrid, 2006}, pages 1393--1403. European Mathematical Society, 2006.

\bibitem[Kon08]{Konyagin08}
Sergey~V. Konyagin.
\newblock On uniformly convergent rearrangements of trigonometric {F}ourier
  series.
\newblock {\em J. Math. Sci.}, 155:81--88, 2008.

\bibitem[K{\"o}r96]{Korner96}
T.~W. K{\"o}rner.
\newblock Divergence of decreasing rearranged {F}ourier series.
\newblock {\em Ann. of Math.}, 144:167--180, 1996.

\bibitem[K{\"o}r99]{Korner99}
T.~W. K{\"o}rner.
\newblock Decreasing rearranged {F}ourier series.
\newblock {\em J. Fourier Anal. Appl.}, 5:1--19, 1999.

\bibitem[KT03]{KonyaginT03}
S.~V. Konyagin and V.~N. Temlyakov.
\newblock Convergence of greedy approximation ii. the trigonometric system.
\newblock {\em Studia Math.}, 159:161--184, 2003.

\bibitem[{\L}W02]{MR1929508}
Izabella {\L}aba and Yang Wang.
\newblock On spectral {C}antor measures.
\newblock {\em J. Funct. Anal.}, 193(2):409--420, 2002.

\bibitem[Rud87]{Rud87}
Walter Rudin.
\newblock {\em Real and complex analysis}.
\newblock McGraw-Hill Book Co., New York, third edition, 1987.

\bibitem[Str00]{MR1785282}
Robert~S. Strichartz.
\newblock Mock {F}ourier series and transforms associated with certain {C}antor
  measures.
\newblock {\em J. Anal. Math.}, 81:209--238, 2000.

\bibitem[Str06]{MR2279556}
Robert~S. Strichartz.
\newblock Convergence of mock {F}ourier series.
\newblock {\em J. Anal. Math.}, 99:333--353, 2006.

\bibitem[Vin85]{Vino85}
I.~M. Vinogradov.
\newblock The method of trigonometric sums in number theory.
\newblock In {\em Selected works}, pages 183--185. Springer-Verlag, 1985.

\bibitem[Wan99]{Wan99}
Yang Wang.
\newblock Self-affine tiles.
\newblock In {\em Advances in wavelets ({H}ong {K}ong, 1997)}, pages 261--282.
  Springer, Singapore, 1999.

\bibitem[Zyg02]{Zyg02}
A.~Zygmund.
\newblock {\em Trigonometric series. {V}ol. {I}, {II}}.
\newblock Cambridge Mathematical Library. Cambridge University Press,
  Cambridge, third edition, 2002.
\newblock With a foreword by Robert A. Fefferman.

\end{thebibliography}

\end{document}